\documentclass{amsart}

\usepackage{float,mathtools,placeins,amsthm,amsmath,amssymb,nicefrac}
\usepackage{hyperref,nameref}
\usepackage[shortlabels]{enumitem}
\RequirePackage{fix-cm}

\usepackage{biblatex}
\newtheorem{theorem}{Theorem}[section]
\newtheorem{proposition}[theorem]{Proposition}
\newtheorem{lemma}[theorem]{Lemma}

\numberwithin{equation}{section}

\begin{document}

\title{The Classification of Supersingular Elliptic Curves in characteristic 3}
\author[A. Orlov] {Alexey Orlov}

\begin{abstract}
We provide an explicit classification of supersingular elliptic curves in characteristic~3 into isomorphism classes, and give explicit formulae for their point counts. This report was written specifically to support implementation of point counting in Hecke.jl~\cite{Nemo}; accordingly, the reasoning is concrete and elementary. After the first version appeared, we learned that the explicit classification was already obtained by Morain~\cite{Morain97}. We believe the present note remains useful in offering a more self-contained and implementation-oriented treatment. The explicit connection to~\cite{Morain97} is given in the appendix.
\end{abstract}

\maketitle

\section{Canonical Form}

We fix $q=3^d$ and consider an elliptic curve
\[E/\mathbb{F}_q \colon y^2 + a_1xy + a_3y = x^3 + a_2x^2 + a_4x + a_6.\]
Completing the square via the substitution
\[y \mapsto \frac{y - a_1x - a_3}{2},\]
we obtain
\[y^2 = 4x^3 + b_2x^2 + 2b_4x + b_6,\]
where, as usual,
\begin{equation*}\begin{split}
b_2 &= a_1^2 + 4a_2,    \\
b_4 &= 2a_4 + a_1a_3,   \\
b_6 &= a_3^2 + 4a_6.
\end{split}\end{equation*}
In characteristic $3$, this simplifies to
\[y^2 = x^3 + b_2x^2 - b_4x + b_6,\]
with
\[j = \frac{(b_2^2 - 24b_4)^3}{\Delta} = \frac{b_2^6}{\Delta}.\]
For supersingular curves in characteristic $3$, we have $j=0$ (see, e.g., \cite[Ex. V.4.5]{Silverman_2010}) and hence $b_2=0$. We thus obtain the canonical form
\[E \colon y^2 = x^3 + a_4x + a_6,\]
with discriminant $\Delta = -a_4^3 \neq 0$, so that $a_4 \neq 0$.

We now determine when two curves of this form are isomorphic over $\mathbb{F}_q$; for the general transformation formulae, see \cite[Table 3.1]{Silverman_2010}. Consider two curves
\[E \colon y^2 = x^3 + a_4x + a_6,\qquad E' \colon y^2 = x^3 + a_4'x + a_6'.\]
Since $a_1 = a_1' = 0$, the equation
\[ua_1' = a_1 + 2s\]
yields $s=0$. Similarly, since $a_3 = a_3' = 0$, the equation
\[u^3a_3' = a_3 + ra_1 + 2t\]
yields $t=0$. Thus the general form of an isomorphism is $(x,y) = (u^2x' + r, u^3y')$ with $u \in \mathbb{F}_q^{\times}$ and $r \in \mathbb{F}_q$ such that
\begin{equation}\label{eq:isomorphism_supersingular3}\begin{split}
u^4a_4' &= a_4  \\
u^6a_6' &= a_6 + ra_4 + r^3
\end{split}\end{equation}

\subsection{Notation}

We fix the following notation throughout.

When $d$ is even, $-1$ is a square in $\mathbb{F}_q$ (see Section~\ref{sub:even_d_classification}); we denote its square roots by $\pm \tau$. We fix an element $\alpha \in \mathbb{F}_q$ with $\operatorname{Tr}(\alpha) = 1$. Finally we fix a primitive root $\beta$ of $\mathbb{F}_q^{\times}$.

\section{Classification of Supersingular Curve Types}

\subsection{\texorpdfstring{$d$}{d} odd}

When $d$ is odd, since $\gcd(q-1,4) = 2$, the polynomial $x^4-1$ has exactly two roots in $\mathbb{F}_q$, namely $\pm1$. In particular, $-1$ is not a square in $\mathbb{F}_q$.

Moreover, the restriction of the map $x \mapsto x^4$ to $(\mathbb{F}_q^{\times})^2$ has trivial kernel, hence it is injective and thus surjective, since $(\mathbb{F}_q^{\times})^2$ is finite. That is, for odd $d$, an element of $\mathbb{F}_q^{\times}$ is a square if and only if it is a fourth power. Since $-1$ is not a square, exactly one of $a_4$ and $-a_4$ is a square, and hence a fourth power.

We thus consider two types of curves:
\begin{itemize}
    \item{Type \textrm{I}: $-a_4$ is a fourth power.}
    \item{Type $\textrm{I}^{+}$: $a_4$ is a fourth power.}
\end{itemize}

\subsection{\texorpdfstring{$d$}{d} even}\label{sub:even_d_classification}

When $d$ is even, we have $q \equiv 1 \pmod{8}$. Since $-1 = \beta^{(q-1)/2}$, we conclude that $-1$ is a fourth power. The cosets of $(\mathbb{F}_q^\times)^4$ in $\mathbb{F}_q^\times$ are
\[(\mathbb{F}_q^\times)^4, \beta(\mathbb{F}_q^\times)^4, \beta^2(\mathbb{F}_q^\times)^4, \beta^3(\mathbb{F}_q^\times)^4.\]
Since $-1$ is a fourth power, both $a_4$ and $-a_4$ lie in the same coset. We thus consider three curve types:
\begin{itemize}
    \item{Type \textrm{I}: $-a_4$ is a fourth power.}
    \item{Type \textrm{II}: $-a_4$ is a square but not a fourth power.}
    \item{Type \textrm{III}: $-a_4$ is not a square.}
\end{itemize}

\subsection{Type \textrm{I} Curves}\label{sub:type_1_classification}

Suppose that $-a_4 = v^4$ for some $v \in \mathbb{F}_q^{\times}$. Applying the substitution $(x,y) \mapsto (v^{-2}x,v^{-3}y)$, we obtain the equation
\[E \colon y^2 = x^3 - x + b,\]
with $b = a_6v^{-6}$. Suppose
\[E' \colon y^2 = x^3 + a_4'x + a_6'\]
is isomorphic to $E$ over $\mathbb{F}_q$. By \eqref{eq:isomorphism_supersingular3}, we have
\begin{equation*}\begin{split}
u^4a_4' &= -1   \\
u^6a_6' &= b + r^3 + r
\end{split}\end{equation*}
Hence $-a_4' = u^{-4}$, which is a fourth power in $\mathbb{F}_q^{\times}$, so $E'$ is also of type \textrm{I}.

Suppose that two curves
\[E_b \colon y^2 = x^3 - x + b, \textnormal{ and } E_{b'} \colon y^2 = x^3 - x + b'\]
are isomorphic over $\mathbb{F}_q$. By \eqref{eq:isomorphism_supersingular3}, we have
\begin{equation*}\begin{split}
u^4 &= 1 \\
r^3 - r + (b-u^6b') &= 0
\end{split}\end{equation*}
The second equation is an Artin--Schreier equation in $r$, which has a solution in $\mathbb{F}_q$ if and only if
\[\operatorname{Tr}(b-u^6b') = \operatorname{Tr}(b-u^2b') = 0,\]
by the additive Hilbert Theorem 90 (see, e.g., \cite[Prop. 3.1.25]{Cohen_2007}).

When $d$ is odd, we have seen that $u^4 = 1$ implies $u^2 = 1$. Hence the isomorphism class of $E_b$ depends only on $\operatorname{Tr}(b) \in \mathbb{F}_3$, giving exactly three isomorphism classes. Moreover, since the trace is $\mathbb{F}_3$-linear, we have $\operatorname{Tr}(-b) = -\operatorname{Tr}(b)$. Thus for $b$ with non-zero trace, $E_b$ is the quadratic twist of $E_{-b}$ by $-1$. The classification is then as follows:
\begin{enumerate}[(i)]
    \item If $\operatorname{Tr}(b) = 0$, then $E_b$ is isomorphic to $E_0$.
    \item If $\operatorname{Tr}(b) = 1$, then $E_b$ is isomorphic to $E_{\alpha}$.
    \item If $\operatorname{Tr}(b) = -1$, then $E_b$ is a quadratic twist of $E_{\alpha}$.
\end{enumerate}

When $d$ is even, $-1$ is a square in $\mathbb{F}_q$, so $E_b$ and $E_{-b}$ are isomorphic (take $u=\tau$ in \eqref{eq:isomorphism_supersingular3}). Hence there are exactly two isomorphism classes:
\begin{enumerate}[(i)]
    \item If $\operatorname{Tr}(b) = 0$, then $E_b$ is isomorphic to $E_0$.
    \item If $\operatorname{Tr}(b) \neq 0$, then $E_b$ is isomorphic to $E_{\alpha}$.
\end{enumerate}

\subsection{Type \texorpdfstring{$\textrm{I}^{+}$}{\textrm{I}+} Curves}

By an analogous argument, every curve of Type $\textrm{I}^{+}$ is isomorphic to
\[E^{+} \colon y^2 = x^3 + x + b.\]
Suppose that two curves
\[E_b \colon y^2 = x^3 + x + b, \textnormal{ and } E_{b'} \colon y^2 = x^3 + x + b'\]
are isomorphic over $\mathbb{F}_q$. By \eqref{eq:isomorphism_supersingular3}, we have
\begin{equation*}\begin{split}
u^4 &= 1 \\
r^3 + r + (b-u^6b') &= 0
\end{split}\end{equation*}
Since $d$ is odd, $-1$ is not a square, thus the map $x \mapsto x^3+x$ has trivial kernel and gives a bijection on $\mathbb{F}_q$. Thus a solution $r \in \mathbb{F}_q$ exists for any $b, b' \in \mathbb{F}_q$. Therefore all Type $\textrm{I}^{+}$ curves are isomorphic to
\[E \colon y^2 = x^3 + x.\]

\subsection{Type \textrm{II} Curves}

As we will see in the Section~\ref{sub:type_2_point_count}, every elliptic curve of Type \textrm{II} is a quadratic twist of the Type \textrm{I} curve. For point counting, the classification is not needed, but we include it here for completeness.

For Type \textrm{II} curves, we have $a_4 \in \beta^2(\mathbb{F}_q^{\times})^4$. There exists $u \in \mathbb{F}_q^{\times}$ such that $u^4 = \beta^2(-a_4)^{-1}$. Applying the substitution $(x,y) \mapsto (u^2x, u^3y)$, we obtain
\[E \colon y^2 = x^3 - \beta^2 x + b\]
with $b = u^6a_6$. From \eqref{eq:isomorphism_supersingular3}, any elliptic curve isomorphic to $E$ over $\mathbb{F}_q$ is also of Type \textrm{II}.

Suppose that two curves
\[E_b \colon y^2 = x^3 - \beta^2 x + b, \textnormal{ and } E_{b'} \colon y^2 = x^3 - \beta^2 x + b'\]
are isomorphic over $\mathbb{F}_q$. By \eqref{eq:isomorphism_supersingular3}, we have
\begin{equation*}\begin{split}
u^4 &= 1   \\
r^3 - \beta^2r + (b - u^6b') &= 0   \\
\end{split}\end{equation*}
Thus $u \in \{\pm1, \pm\tau\}$, so $u^6 = u^2 = \pm1$, and the second equation becomes $r^3 - \beta^2 r + (b \pm b') = 0$. Substituting $r = \beta z$ we obtain
\[z^3 - z + \frac{b \pm b'}{\beta^3} = 0.\]
This is an Artin--Schreier equation, which has a solution in $\mathbb{F}_q$ if and only if the trace of the constant term vanishes, which is equivalent to $\operatorname{Tr}(b\beta^{-3}) = \pm \operatorname{Tr}(b'\beta^{-3})$. Hence there are exactly two isomorphism classes:
\begin{enumerate}[(i)]
    \item If $\operatorname{Tr}(b\beta^{-3}) = 0$, then $E_b$ is isomorphic to $E_0$.
    \item If $\operatorname{Tr}(b\beta^{-3}) \neq 0$, then $E_b$ is isomorphic to $E_{\alpha\beta^3}$.
\end{enumerate}

\subsection{Type \textrm{III} Curves}

For Type \textrm{III} curves, we have $a_4 \in \beta(\mathbb{F}_q^{\times})^4 \cup \beta^3(\mathbb{F}_q^{\times})^4$. We will distinguish these two cases.

\subsubsection{Type \textrm{III}a Curves}

These are the curves with $a_4 \in \beta(\mathbb{F}_q^{\times})^4$. There exists $u \in \mathbb{F}_q^{\times}$ such that $u^4 = \beta(-a_4)^{-1}$. Applying the substitution $(x,y) \mapsto (u^2x, u^3y)$, we obtain
\[E \colon y^2 = x^3 - \beta x + b\]
with $b = u^6a_6$. From \eqref{eq:isomorphism_supersingular3}, any elliptic curve isomorphic to $E$ over $\mathbb{F}_q$ is also of Type \textrm{III}a.

Suppose that two curves
\[E_b \colon y^2 = x^3 - \beta x + b, \textnormal{ and } E_{b'} \colon y^2 = x^3 - \beta x + b'\]
are isomorphic over $\mathbb{F}_q$. By \eqref{eq:isomorphism_supersingular3}, we have
\begin{equation*}\begin{split}
u^4 &= 1   \\
r^3 - \beta r + (b - u^6b') &= 0   \\
\end{split}\end{equation*}
Thus $u \in \{\pm1, \pm\tau\}$, so $u^6 = u^2 = \pm1$, and the second equation becomes $r^3 - \beta r + (b \pm b') = 0$. Since $\beta$ is a primitive root, and we are in odd characteristic, $\beta$ is not a square, hence the map $x \mapsto x^3 - \beta x$ has trivial kernel giving a bijection on $\mathbb{F}_q$. Hence a solution $r \in \mathbb{F}_q$ always exists, and all Type \textrm{III}a curves are isomorphic to $E \colon y^2 = x^3 - \beta x$.

\subsubsection{Type \textrm{III}b Curves}

These are the curves with $a_4 \in \beta^3(\mathbb{F}_q^{\times})^4$. By an analogous argument, we obtain the canonical form
\[E \colon y^2 = x^3 - \beta^3 x + b.\]
As before, since $\beta^3$ is not a square, the map $x \mapsto x^3 - \beta^3 x$ has trivial kernel. Thus $r^3 - \beta^3 r + (b - u^6b') = 0$ always has a solution, hence all Type \textrm{III}b curves are isomorphic to $E \colon y^2 = x^3 - \beta^3 x$.

\section{Counting Points}

We recall the following standard formula. For an elliptic curve $E \colon y^2 = f(x)$ over a finite field $\mathbb{F}_q$ of odd characteristic, we have
\begin{equation}\label{eq:curve_order_chi}
\#E(\mathbb{F}_q) = q+1 + \sum_{x \in \mathbb{F}_q}\chi(f(x)),
\end{equation}
where $\chi: \mathbb{F}_q^{\times} \to \{\pm1\}$ denotes the quadratic character, extended by $\chi(0) = 0$.

\subsection{Type \textrm{I} Curves}

Recall that every curve of type \textrm{I} is isomorphic to one of the form
\[E \colon y^2 = x^3 - x + b.\]
To count points on it, we require the following result:
\begin{lemma}\label{lem:sum_char_fixed_trace}
Let $q=3^d$. For $a \in \mathbb{F}_3$, define
\[S_d(a) = \sum_{\substack{x \in \mathbb{F}_q \\ \operatorname{Tr}(x) = a}}\chi(x).\]
For $d$ odd,
\[S_d(a) = \begin{cases}
0                                       & \text{if } a = 0, \\
(-1)^{\frac{d-1}{2}} 3^{\frac{d-1}{2}}  & \text{if } a = 1, \\
(-1)^{\frac{d+1}{2}} 3^{\frac{d-1}{2}}  & \text{if } a = -1.
\end{cases}\]
For $d$ even,
\[S_d(a) = \begin{cases}
2 \cdot (-1)^{\frac{d+2}{2}} 3^{\frac{d-2}{2}}  & \text{if } a = 0, \\
(-1)^{\frac{d}{2}} 3^{\frac{d-2}{2}}          & \text{if } a = \pm1.
\end{cases}\]
\end{lemma}
\begin{proof}
Let $\zeta_3=e^{2\pi i/3}$ be the primitive cube root of unity, and define the indicator function
\[g_a(x) = \frac{1}{3}\sum_{k \in \mathbb{F}_3}\zeta_3^{k(\operatorname{Tr}(x)-a)}.\]
By orthogonality of additive characters, this function equals one when $\operatorname{Tr}(x) = a$ and zero otherwise. Hence
\[S_d(a) = \frac{1}{3}\sum_{x \in \mathbb{F}_q} \chi(x)\sum_{k \in \mathbb{F}_3}\zeta_3^{k(\operatorname{Tr}(x)-a)}.\]
Interchanging the order of summation yields
\[S_d(a) = \frac{1}{3}\sum_{k \in \mathbb{F}_3}\zeta_3^{-ka}\sum_{x \in \mathbb{F}_q}\chi(x)\zeta_3^{k\operatorname{Tr}(x)}
    = \frac{1}{3}\sum_{k \in \mathbb{F}_3}\zeta_3^{-ka}\sum_{x \in \mathbb{F}_q}\chi(x)\zeta_3^{\operatorname{Tr}(kx)},\]
where the last equality is by the $\mathbb{F}_3$-linearity of trace.

For $k=0$, the inner sum vanishes. For $k \neq 0$, the substitution $y=kx$ gives
\[\chi(k)^{-1}\sum_{y \in \mathbb{F}_q}\chi(y)\zeta_3^{\operatorname{Tr}(y)} = \chi(k)G(\chi),\]
where $G(\chi)$ is the Gauss sum associated to the canonical additive character. Thus
\begin{equation}\label{eq:sum_char_fixed_trace}
S_d(a) = \frac{1}{3}G(\chi)\left(\zeta_3^{-a} + \zeta_3^{a}\chi(-1)\right).
\end{equation}

By a classical result on Gauss sums (see, e.g., \cite[Th. 5.15]{Lidl_1996}), we have
\[G(\chi) = (-1)^{d-1}i^d\sqrt{q}.\]

For $d$ odd, we have $\chi(-1)=-1$, and \eqref{eq:sum_char_fixed_trace} becomes
\[S_d(a) = \frac{1}{3}i^d\sqrt{q}(\zeta_3^{-a} - \zeta_3^{a}).\]
We have $S_d(0) = 0$ and $S_d(-1) = -S_d(1)$. We compute
\[S_d(1) = -\frac{1}{3}i^{d+1}\sqrt{q}\sqrt{3} = (-1)^{\frac{d-1}{2}} 3^{\frac{d-1}{2}}.\]

For $d$ even, we have $\chi(-1)=1$, and \eqref{eq:sum_char_fixed_trace} becomes
\[S_d(a) = -\frac{1}{3}i^d\sqrt{q}(\zeta_3^{-a} + \zeta_3^{a}).\]
We compute
\[S_d(0) = -\frac{2}{3}i^d\sqrt{q} = 2 \cdot (-1)^{\frac{d+2}{2}} 3^{\frac{d-2}{2}},\]
and
\[S_d(1) = S_d(-1) = \frac{1}{3}i^d\sqrt{q} = (-1)^{\frac{d}{2}} 3^{\frac{d-2}{2}}.\]
\end{proof}

We now apply Lemma~\ref{lem:sum_char_fixed_trace} to count points on Type \textrm{I} curves. We need to compute $\sum_{x \in \mathbb{F}_q}\chi(x^3-x-b)$. The kernel of the Artin--Schreier map $x \mapsto x^3-x$ is $\mathbb{F}_3$, and its image is the set
\[H = \{y \in \mathbb{F}_q: \operatorname{Tr}(y) = 0\}\]
consisting of elements with trace zero. Thus
\[\sum_{x \in \mathbb{F}_q}\chi(x^3-x + b) = 3\sum_{\substack{y \in \mathbb{F}_q \\ \operatorname{Tr}(y) = 0}}\chi(y + b)
    = 3\sum_{\substack{y \in \mathbb{F}_q \\ \operatorname{Tr}(y) = \operatorname{Tr}(b)}}\chi(y).\]
Hence
\[\#E(\mathbb{F}_q) = q+1 + 3S_d(\operatorname{Tr}(b)).\]
We can now write the formulae explicitly.
\begin{proposition}\label{prop:point_count_type1}
Let $q=3^d$ and consider an elliptic curve
\[E/\mathbb{F}_q \colon y^2 = x^3 - x + b.\]
For $d$ odd,
\[\#E(\mathbb{F}_q) = \begin{cases}
q+1                                 & \text{if } \operatorname{Tr}(b) = 0, \\
q+1 + (-1)^{\frac{d-1}{2}}\sqrt{3q} & \text{if } \operatorname{Tr}(b) = 1, \\
q+1 - (-1)^{\frac{d-1}{2}}\sqrt{3q} & \text{if } \operatorname{Tr}(b) = -1.
\end{cases}\]
For $d$ even,
\[\#E(\mathbb{F}_q) = \begin{cases}
q+1 - (-1)^{\frac{d}{2}} 2 \sqrt{q} & \text{if } \operatorname{Tr}(b) = 0, \\
q+1 + (-1)^{\frac{d}{2}} \sqrt{q}   & \text{if } \operatorname{Tr}(b) = \pm1.
\end{cases}\]
\end{proposition}

\subsection{Type \texorpdfstring{$\textrm{I}^{+}$}{\textrm{I}+} Curves}

Recall that every curve of type $\textrm{I}^{+}$ is isomorphic to one of the form
\[E \colon y^2 = x^3 + x,\]
which occurs only when $d$ is odd. As shown above, the map $x \mapsto x^3+x$ is a bijection on $\mathbb{F}_q$, hence
\[\sum_{x \in \mathbb{F}_q}\chi(x^3 + x) = \sum_{y \in \mathbb{F}_q}\chi(y) = 0,\]
and we obtain
\begin{proposition}\label{prop:point_count_type1plus}
Let $q=3^d$ with $d$ odd, and consider an elliptic curve
\[E/\mathbb{F}_q \colon y^2 = x^3 + x.\]
Then $\#E(\mathbb{F}_q) = q+1$.
\end{proposition}

\subsection{Type \textrm{II} Curves}\label{sub:type_2_point_count}

For Type \textrm{II} curves, we have $-a_4 \in \beta^2(\mathbb{F}_q^{\times})^4$. Write $-a_4 = \gamma^2$. We see that
\[E \colon y^2 = x^3 + a_4 x + b\]
is a quadratic twist by $\gamma$ of the type \textrm{I} curve
\[E \colon y^2 = x^3 - x + b\gamma^{-3}.\]
From Proposition~\ref{prop:point_count_type1} we immediately obtain
\begin{proposition}\label{prop:point_count_type2}
Let $q=3^d$ with $d$ even, and consider an elliptic curve
\[E/\mathbb{F}_q \colon y^2 = x^3 + a_4 x + b\]
with $-a_4 = \gamma^2$ for some $\gamma \in \mathbb{F}_q^{\times}$ not a square. Then
\[\#E(\mathbb{F}_q) = \begin{cases}
q+1 + (-1)^{\frac{d}{2}} 2 \sqrt{q} & \text{if } \operatorname{Tr}(b\gamma^{-3}) = 0, \\
q+1 - (-1)^{\frac{d}{2}} \sqrt{q}   & \text{if } \operatorname{Tr}(b\gamma^{-3}) = \pm1.
\end{cases}\]
\end{proposition}

\subsection{Type \textrm{III} Curves}

Recall that every curve of Type \textrm{III}a is isomorphic to one of the form
\[E \colon y^2 = x^3 - \beta x.\]
As shown above, $x \mapsto x^3 - \beta x$ is a bijection, hence
\[\sum_{x \in \mathbb{F}_q}\chi(x^3 - \beta x) = \sum_{y \in \mathbb{F}_q}\chi(y) = 0.\]

Similarly, every Type IIIb curve is isomorphic to
\[E \colon y^2 = x^3 - \beta^3 x.\]
Since $x \mapsto x^3 - \beta^3 x$ is also a bijection,
\[\sum_{x \in \mathbb{F}_q}\chi(x^3 - \beta^3 x) = \sum_{y \in \mathbb{F}_q}\chi(y) = 0.\]

We therefore obtain
\begin{proposition}\label{prop:point_count_type3}
Let $q=3^d$ with $d$ even, and let $\beta$ be the fixed primitive root of $\mathbb{F}_q^{\times}$. Consider elliptic curves
\[E/\mathbb{F}_q \colon y^2 = x^3 - \beta x, \textnormal{ and } E'/\mathbb{F}_q \colon y^2 = x^3 - \beta^3 x.\]
Then
\[\#E(\mathbb{F}_q) = \#E'(\mathbb{F}_q) = q+1.\]
\end{proposition}

\section{Algorithm for Counting Points on Supersingular Curves in Characteristic 3}

Let $q=3^d$, and let
\[E \colon y^2 = x^3 + a_4x + a_6\]
be a supersingular elliptic curve.

\paragraph{\textbf{$d$ odd.}}
If $-a_4 = v^4$ for some $v \in \mathbb{F}_q^{\times}$, then
\[\#E(\mathbb{F}_q) = \begin{cases}
q + 1               & \text{if } \operatorname{Tr}(a_6v^{-6}) = 0,  \\
q + 1 + \sqrt{3q}   & \text{if } \operatorname{Tr}(a_6v^{-6}) = 1, \textnormal{ and } d \equiv 1 \pmod{4},  \\
q + 1 - \sqrt{3q}   & \text{if } \operatorname{Tr}(a_6v^{-6}) = 1, \textnormal{ and } d \equiv 3 \pmod{4},  \\
q + 1 - \sqrt{3q}   & \text{if } \operatorname{Tr}(a_6v^{-6}) = -1, \textnormal{ and } d \equiv 1 \pmod{4}, \\
q + 1 + \sqrt{3q}   & \text{if } \operatorname{Tr}(a_6v^{-6}) = -1, \textnormal{ and } d \equiv 3 \pmod{4}.
\end{cases}\]
Otherwise, $\#E(\mathbb{F}_q) = q+1$.

\paragraph{\textbf{$d$ even.}}
If $-a_4 = \gamma^2$ for some $\gamma \in \mathbb{F}_q^{\times}$, and $\gamma$ itself is square, we have
\[\#E(\mathbb{F}_q) = \begin{cases}
q + 1 - 2\sqrt{q}   & \text{if } \operatorname{Tr}(a_6\gamma^{-3}) = 0, \textnormal{ and } d \equiv 0 \pmod{4},      \\
q + 1 + 2\sqrt{q}   & \text{if } \operatorname{Tr}(a_6\gamma^{-3}) = 0, \textnormal{ and } d \equiv 2 \pmod{4},      \\
q + 1 + \sqrt{q}    & \text{if } \operatorname{Tr}(a_6\gamma^{-3}) \neq 0, \textnormal{ and } d \equiv 0 \pmod{4},   \\
q + 1 - \sqrt{q}    & \text{if } \operatorname{Tr}(a_6\gamma^{-3}) \neq 0, \textnormal{ and } d \equiv 2 \pmod{4}.
\end{cases}\]

If $-a_4 = \gamma^2$ for some $\gamma \in \mathbb{F}_q^{\times}$, and $\gamma$ is not a square, we have
\[\#E(\mathbb{F}_q) = \begin{cases}
q + 1 + 2\sqrt{q}   & \text{if } \operatorname{Tr}(a_6\gamma^{-3}) = 0, \textnormal{ and } d \equiv 0 \pmod{4},      \\
q + 1 - 2\sqrt{q}   & \text{if } \operatorname{Tr}(a_6\gamma^{-3}) = 0, \textnormal{ and } d \equiv 2 \pmod{4},      \\
q + 1 - \sqrt{q}    & \text{if } \operatorname{Tr}(a_6\gamma^{-3}) \neq 0, \textnormal{ and } d \equiv 0 \pmod{4},   \\
q + 1 + \sqrt{q}    & \text{if } \operatorname{Tr}(a_6\gamma^{-3}) \neq 0, \textnormal{ and } d \equiv 2 \pmod{4};
\end{cases}\]

Finally, if $-a_4$ is not a square, then $\#E(\mathbb{F}_q) = q+1$.

\appendix\section{Connection to Morain's classification}

In \cite{Morain97} the curves $E_{a,b}^M \colon y^2 = x^3 + ax + b$ are considered. Following Morain we write $\delta$ for some element of trace one, and $\gamma$ for a non-square element.

Our classification corresponds to Morain's as follows:
\begin{itemize}
    \item Type~\textrm{I} curves correspond to $E_{-1,b}^M$. Proposition~\ref{prop:point_count_type1} gives an explicit point-count formula for the curves isomorphic to $E_{-1,\delta}^M$, $E_{-1,-\delta}^M$, $E_{-1,0}^M$ when $d$ is odd, and to $E_{-1,\delta}^M$ and $E_{-1,0}^M$ when $d$ is even.
    \item Type~$\textrm{I}^{+}$ curves are isomorphic to $E_{1,0}^M$.
    \item Type~\textrm{II} curves are isomorphic to $E_{-\gamma^2,0}^M$ and $E_{-\gamma^2,\gamma^3\delta}^M$.
    \item Type~\textrm{IIIa} curves are isomorphic to $E_{-\gamma,0}^M$, and Type~\textrm{IIIb} curves are isomorphic to $E_{-\gamma^3,0}^M$.
\end{itemize}

It is clear that our classification agrees with the one by Morain, and gives the same formulae for the number of points and the trace of Frobenius.

\FloatBarrier
\printbibliography

\end{document}